\documentclass{amsart}

\usepackage{paralist}
\usepackage{color}
\usepackage{mathrsfs}
\usepackage{amssymb}
\usepackage{amsmath}
\usepackage{amsfonts}
\usepackage{stmaryrd}
\usepackage{mathtools}
\usepackage{tikz-cd}

\newtheorem*{thma}{Theorem~A}
\newtheorem*{thmb}{Theorem~B}
\newtheorem*{thmc}{Theorem~C}

\newtheorem{theorem}{Theorem}[section]

\newtheorem{proposition}[theorem]{Proposition}

\newtheorem{claim}[theorem]{Claim}
\newtheorem{subclaim}[theorem]{Subclaim}

\newtheorem{question}[theorem]{Question}

\theoremstyle{definition}
\newtheorem{definition}[theorem]{Definition}
\newtheorem{remark}[theorem]{Remark}
\newtheorem{example}[theorem]{Example}

\usepackage[pdftex]{hyperref}
\hypersetup{
    colorlinks=true, 
    linktoc=all,     
    linkcolor=blue,  
    allcolors=blue,
}

\usepackage[framemethod=tikz]{mdframed}

\newcommand{\dom}{\mathrm{dom}}
\newcommand{\bb}{\mathbb}

\newcommand{\mc}{\mathcal}

\newcommand{\ra}{\rightarrow}
\newcommand{\cl}{\mathrm{cl}}

\newcommand{\ZFC}{\sf ZFC}

\newcommand{\CH}{\sf CH}

\title{Generalized almost disjoint families and injective Banach spaces}
\author{Chris Lambie-Hanson}
\address[Lambie-Hanson]{
Institute of Mathematics, 
Czech Academy of Sciences, 
{\v Z}itn{\'a} 25, 110 00 Praha 1, Czech Republic
}
\email{lambiehanson@math.cas.cz}
\urladdr{https://clambiehanson.github.io/}

\author{David Schrittesser}
\address[Schrittesser]{Institute for Advanced Study in Mathematics, 
Harbin Institute of Technology,
92 West Da Zhi Street,
Harbin, Heilongjiang 150001, China}
\email{david.schrittesser@univie.ac.at}


\thanks{The first author was supported by the Czech Academy of Sciences 
(RVO 67985840) and the GA\v{C}R project 23-04683S. We thank Jeffrey Bergfalk for helpful 
comments on an earlier draft.}

\begin{document}

\begin{abstract}
  A fundamental open problem in the homological theory of Banach spaces 
  is the calculation of the injective dimension of the Banach space $c_0$. We make a 
  contribution to the study of this problem by proving that, if the 
  Continuum Hypothesis ($\mathsf{CH}$) holds, 
  then the injective dimension of $c_0$ is at 
  least 3. In the course of proving this result, we introduce the notion of 
  an \emph{almost disjoint family} on a topological space $X$, generalizing 
  the classical notion of almost disjoint families of subsets of 
  $\mathbb{N}$, which we feel is of interest in its own right. 
  We prove that, if $\mathfrak{b} = 2^{\aleph_0}$, then there exists an almost disjoint 
  family of cardinality $2^{\aleph_1}$ on the \v{C}ech-Stone remainder of 
  $\mathbb{N}$.
\end{abstract}

\maketitle

\section{Introduction}

Almost disjoint families of subsets of $\bb{N}$, which are fundamental objects 
in combinatorial set theory, have a long and fruitful history of 
applications to the study of Banach spaces (cf.\ \cite[\S 9]{hrusak_ad}). 
In this article, we introduce a broad generalization of the notion of 
``almost disjoint family" and apply it to the study of the question of the 
injective dimension of $c_0$. 

In order to properly introduce our results, let us begin by briefly reviewing the relevant definitions around injective
Banach spaces (see \cite{separably_injective} and 
\cite{sanchez_homological} for all of the 
facts and definitions presented here, plus much more).
A Banach space $E$ is \emph{injective} if, whenever $X$ and $Y$ are Banach 
spaces with $X \subseteq Y$, every operator $t:X \ra E$ can be extended to 
an operator $T:Y \ra E$. There are a number of equivalent characterizations; 
for example, $E$ is injective if and only if it is isomorphic to a complemented 
subspace of $\ell_\infty(\Gamma)$ for some set $\Gamma$. It then follows 
that every Banach space embeds (isometrically) as a (closed) subspace of an 
injective Banach space. This is immediate for finite-dimensional Banach 
spaces, which are themselves injective, and given an infinite-dimensional 
Banach space $X$ and any dense subset $\Gamma \subseteq X$, $X$ is isometric 
to a subspace of $\ell_\infty(\Gamma)$.
 
Given a Banach space $X$, an \emph{injective resolution} of $X$ is an exact 
sequence of the form 
\[
  0 \ra X \ra I_0 \ra I_1 \ra I_2 \ra \cdots
\]
in which $I_n$ is injective for every $n \in \mathbb{N}$. 
The observations of the previous paragraph allow us to form injective 
resolutions of arbitrary Banach spaces as follows. For 
convenience, fix for each Banach space $X$ an injective Banach space 
$E(X)$ and an isometric embedding $\iota(X) : X \ra E(X)$. Now, given a Banach 
space $X$, we can form the following commutative diagram:
\begin{center}
\begin{tikzcd}[arrows = {-stealth}, column sep = {small}]
  0 \ar[r] & X \ar[r, "\iota(X)"] & E(X) = I_0 \ar[rr] \ar[dr, "\pi_0"] && E(c\kappa(X)) 
  = I_1 \ar[dr, "\pi_1"] \ar[rr] && I_2 \dots \\
  &&& c\kappa(X) \ar[ur, "\iota(c\kappa(X))"] 
  && c\kappa^2(X) \ar[ur, "\iota(c\kappa^2(X))"]
\end{tikzcd}
\end{center}
where each map $\pi_n$ is the relevant quotient map; so, e.g., we have 
$c\kappa(X) = I_0/\mathrm{im}(\iota(X))$ and $c\kappa^2(X) = 
I_1/\mathrm{im}(\iota(c\kappa(X))$. The top line of the above diagram is 
then an injective resolution of $X$, $0 \ra X \ra I_0 \ra I_1 \ra I_2 \ra \cdots$. The \emph{injective 
dimension of $X$}, denoted $\mathrm{id}(X)$, is equal to $0$ if 
$X$ is injective. Otherwise, it is the least $n < \omega$ such that 
$c\kappa^n(X)$ is injective, if such an $n$ exists, or $\infty$ if no such 
$n$ exists. It can be shown (cf.\ \cite{sanchez_homological}) 
that $\mathrm{id}(X)$ is independent of our choices 
of $E(Y)$ and $\iota(Y)$, and that $\mathrm{id}(X)$ is equal to the least $n$, 
if it exists, for which there is an injective resolution of $X$ of the form 
\[
0 \ra X \ra I_0 \ra \cdots \ra I_n \ra 0 \ra \cdots.
\]

Very little is known about the injective dimension of most Banach spaces 
(see \cite{sanchez_homological} for a few exceptions). For example, 
the injective dimension of the space $c_0$ is unknown; this is the motivating 
question of this paper. We find this question of interest in its own right, but we also 
came to it via a foundational question in \emph{condensed mathematics} 
(cf.\ \cite{CS1}). To introduce the question, suppose that $S_0$ and $S_1$ 
are extremally disconnected 
compact Hausdorff spaces. It is not hard to see that the condensed abelian groups 
$\mathbb{Z}[\underline{S_i}]$ are both compact projective objects in the category 
$\mathsf{Cond(Ab)}$ of condensed abelian groups. Also, in general, the tensor 
product $\mathbb{Z}[\underline{S_0}] \otimes \mathbb{Z}[\underline{S_1}] = 
\mathbb{Z}[\underline{S_0 \times S_1}]$ is \emph{not} projective 
in $\mathsf{Cond(Ab)}$ (cf.\ \cite[Proposition 3.7]{complex}). Clausen and 
Scholze raised the following question: in this situation, must 
$\mathbb{Z}[\underline{S_0 \times S_1}]$ have finite projective dimension in 
$\mathsf{Cond(Ab)}$? To the best of our knowledge, the question remains open; 
as Clausen and Scholze observed, one way to give a \emph{negative} answer would be 
to show that $c_0$ has infinite injective dimension in the category of Banach spaces 
(\cite{scholze_email}, also cf.\ \cite[Appendix to Lecture III]{complex}).

To begin exploring this question, let us consider the beginning of the 
commutative diagram yielding an injective resolution of $c_0$ (recall that 
$\mathbb{N}^* = \beta \mathbb{N} \setminus \mathbb{N}$ is the \v{C}ech-Stone 
remainder of $\mathbb{N}$ and $\mathbb{N}^*$ has a natural dense set indexed 
by $[\omega]^\omega$, yielding a closed embedding of $C(\mathbb{N}^*)$ 
into $\ell_\infty([\omega]^\omega)$ (see Section \ref{section: ad_family} below)):
\begin{center}
\begin{tikzcd}[arrows = {-stealth}, column sep = {small}]
  0 \ar[r] & c_0 \ar[r] & \ell_\infty \ar[rr] \ar{dr} & & \ell_\infty([\omega]^\omega) \ar[dr] & \dots \\
  & & & \ell_\infty/c_0 \cong C(\mathbb{N}^*) \ar{ur} & & 
  \ell_\infty([\omega]^\omega)/C(\mathbb{N}^*)
\end{tikzcd}
\end{center}
Phillips \cite{phillips} and Sobczyk \cite{sobczyk} proved around 1940 that $c_0$ 
is not an injective Banach space. In the 1960s, Amir \cite{amir} proved that 
$\ell_\infty/c_0$ is not injective, and hence $\mathrm{id}(c_0) \geq 2$.
Here, we provide a modest improvement to these results under the additional 
assumption of the Continuum Hypothesis ($\CH$).

\begin{thma}
  Suppose that $\CH$ holds. Then $\mathrm{id}(c_0) \geq 3$.
\end{thma}

In the course of proving Theorem A, we introduce and begin an analysis of 
the notion of an \emph{almost disjoint family} on a topological space 
$X$, generalizing the classical notion of an almost disjoint family of 
infinite sets of natural numbers. This notion seems of interest in its own 
right and worthy of further investigation. 

Given a compact Hausdorff space $K$, the 
existence of large almost disjoint families on $K$ has implications regarding 
the injective dimension of the Banach space $C(K)$. This should not be 
terribly surprising, as elegant proofs of the aforementioned results of 
Phillips, Sobczyk, and Amir, due to Whitley and Rosenthal, respectively, 
utilize the existence of classical almost disjoint 
families of cardinality $2^{\aleph_0}$ (cf.\ \cite{whitley},
\cite{rosenthal} and \cite[Theorem 1.25]{separably_injective}).

To be slightly more precise, suppose that $K$ is a compact Hausdorff space, and let 
$D$ be a dense subset of $K$. Then there is a natural embedding of 
$C(K)$ as a closed subspace of $\ell_\infty(D)$. Therefore, the following 
is the beginning of a commutative diagram yielding an injective resolution of 
$C(K)$:
\begin{center}
\begin{tikzcd}[arrows = {-stealth}, column sep = {small}]
  0 \ar[r] & C(K) \ar[r] & \ell_\infty(D) \ar[dr] & \dots \\ 
  &&& \ell_\infty(D)/C(K)
\end{tikzcd}
\end{center}
We establish the following connection between almost disjoint families and 
injective spaces:

\begin{thmb}
  Suppose that $K$ is a compact Hausdorff space, $D \subseteq K$ is dense, 
  and there is an almost disjoint family on $K$ of cardinality $\lambda$ 
  such that $2^\lambda > 2^{|D|}$. Then 
  $\ell_\infty(D)/C(K)$ is not injective, and hence $\mathrm{id}(C(K)) \geq 2$.
\end{thmb}

Theorem A will then follow from Theorem B together with the following result, 
using $\mathfrak{b} = 2^{\aleph_0}$ (a consequence of $\CH$) to construct a specific 
almost disjoint family.

\begin{thmc}
  Suppose that $\mathfrak{b} = 2^{\aleph_0}$. Then there is an almost disjoint family on 
  $\mathbb{N}^*$ of cardinality $2^{\aleph_1}$.
\end{thmc}

The structure of the remainder of the paper is as follows. In Section 
\ref{section: ad_family}, we introduce our generalization of ``almost disjoint 
families" and prove Theorem B. In Section \ref{section: n-star}, we prove 
Theorem C, which then immediately yields Theorem A. We close the paper in 
Section \ref{section: questions} with a few remaining open questions.

\subsection{Notation}

We use $\omega$ and $\mathbb{N}$ interchangeably to denote the set of natural 
numbers. We let $[\omega]^\omega$ denote the set of all infinite subsets of 
$\omega$.
If $x, y \in [\omega]^\omega$, then $x \subseteq^* y$ is the assertion that $x$ is a subset of 
$y$ \emph{mod finite}, i.e., $|x \setminus y| < \aleph_0$. In particular, $x \subsetneq^* y$ is 
the assertion that $|x \setminus y| < \aleph_0$ and $|y \setminus x| = \aleph_0$. Similarly, 
$x =^* y$ is the assertion that $x$ and $y$ are equal mod finite, i.e., 
$|x \triangle y| < \aleph_0$.

Given a nonempty set $\Gamma$, $c_0(\Gamma)$ is the Banach space of all functions 
$f:\Gamma \rightarrow \mathbb{R}$ converging to zero, i.e., all functions $f$ such 
that, for all $\varepsilon > 0$, the set $\{\alpha \in \Gamma \mid |f(\alpha)| > \varepsilon\}$ 
is finite. $\ell_\infty(\Gamma)$ is the Banach space of all bounded functions from 
$\Gamma$ to $\bb{R}$ (both $c_0(\Gamma)$ and $\ell_\infty(\Gamma)$ are given the 
sup-norm). $c_0$ and $\ell_\infty$ denote $c_0(\mathbb{N})$ and 
$\ell_\infty(\mathbb{N})$, 
respectively. If $K$ is a compact Hausdorff space, then $C(K)$ denotes the 
Banach space of all continuous functions from $K$ to $\bb{R}$, again with the 
sup-norm.

\section{Generalized almost disjoint families} \label{section: ad_family}

Recall that a set $\mc{A} \subseteq [\omega]^\omega$ is called an 
\emph{almost disjoint family} if $A \cap B$ is finite for all distinct 
$A,B \in \mc{A}$. Almost disjoint families are fundamental objects in 
combinatorial set theory and have proven to be of considerable utility in 
general topology and functional analysis. We begin this section by introducing a 
generalization of this notion.

\begin{definition} \label{ad_def}
  Suppose that $X$ is a topological space. A family $\mc A$ is an \emph{almost disjoint family 
  on $X$} if
  \begin{enumerate}
    \item every element of $\mc A$ is a regular open subset of $X$ 
    that is not compact;
    \item for all distinct $U_0, U_1 \in \mc A$, $U_0 \cap U_1$ is compact.
  \end{enumerate}
\end{definition}

\begin{example}
  Suppose that $\mc A$ is an almost disjoint family of subsets of $\omega$, 
  i.e., $\mc A$ is a collection of infinite (and co-infinite) subsets of $\omega$ with pairwise finite intersection. Then, as we leave the reader to verify, $\mc A$ is 
  also an almost disjoint family in the sense of Definition \ref{ad_def} on the space $\omega + 1$ (i.e., 
  the convergent sequence).
\end{example}

\begin{remark}
  We expect that, at least among compact Hausdorff spaces, Definition \ref{ad_def} 
  will be of most interest for spaces 
  $X$ that are totally disconnected but not extremally disconnected. Indeed, if 
  $X$ is Hausdorff and $\mc{A}$ is an almost disjoint family on $X$, then 
  $U_0 \cap U_1$ is a clopen subset of $X$ for all distinct $U_0,U_1 \in 
  \mc{A}$. The existence of interesting 
  almost disjoint families on $X$ would therefore entail many clopen subsets of 
  $X$. It then seems natural to focus on spaces $X$ that are totally disconnected, 
  i.e., spaces that have a basis of clopen sets. 
  
  On the other hand, if $X$ is an extremally disconnected compact Hausdorff 
  space, then all regular open subsets of $X$ are closed and hence compact, and 
  therefore there are \emph{no} nonempty almost disjoint families on $X$. 
  This makes sense in the context of Theorem B, since, if $X$ is an extremally 
  disconnected compact Hausdorff space, then $C(X)$ is an injective Banach space 
  (cf.\ \cite[\S 1.3]{separably_injective}).
\end{remark}

A key ingredient in our proof of Theorem B will be the following 
result of Rosenthal \cite{rosenthal}.

\begin{theorem} \label{thm: rosenthal}
  Suppose that $E$ is an injective Banach space and $E$ contains an isomorphic 
  copy of $c_0(\Gamma)$ for some set $\Gamma$. Then $E$ contains an isomorphic 
  copy of $\ell_\infty(\Gamma)$. \qed
\end{theorem}

In particular, if one wants to show that a certain Banach space $X$ is not injective, 
one possible strategy is to prove that there is some set $\Gamma$ such that 
$X$ contains a copy of $c_0(\Gamma)$ but $\Gamma$ is large enough so that 
$X$ \emph{cannot} contain a copy of $\ell_\infty(\Gamma)$. To achieve this, the 
following theorem, also due to Rosenthal, will be useful.

\begin{theorem}[{\cite[Theorem 3.4]{rosenthal}}] \label{thm: rosenthal_2}
  Suppose that $X$ is a Banach space, $\Gamma$ is an infinite set, 
  $\langle x_\alpha \mid \alpha \in \Gamma \rangle$ is a sequence of elements of 
  $X$, and $C > 0$ are such that
  \begin{enumerate}
    \item for all $\alpha \in \Gamma$, we have $\|x_\alpha\| = 1$;
    \item for all positive $n < \omega$, all real numbers 
    $\langle r_i \mid i < n \rangle$, and all pairwise distinct elements 
    $\langle \alpha_i \mid i < n \rangle$ from $\Gamma$, we have
    \[
      \left\| \sum_{i < n} r_i x_{\alpha_i} \right\| \leq C 
      \sup_{i < n} |r_i|.
    \]
  \end{enumerate}
  Then there is $\Gamma' \subseteq \Gamma$ such that $|\Gamma'| = |\Gamma|$ and 
  the closure of the span of $\{x_\alpha \mid \alpha \in \Gamma'\}$ in 
  $X$ is isomorphic to $c_0(\Gamma')$.
\end{theorem}

Now suppose that $K$ is a compact Hausdorff space and $D \subseteq K$ is dense. Then there is a closed 
embedding of $C(K)$ into $\ell_\infty(D)$ given by sending each $\varphi \in C(K)$ to 
$\langle \varphi(x) \mid x \in D \rangle$. When we write $\ell_\infty(D)/C(K)$, it is implicit that 
the quotient is with respect to this embedding.

\begin{proposition} \label{prop: c_0_copy}
  Suppose that $\lambda$ is an infinite cardinal, $K$ is a compact Hausdorff space, $D \subseteq K$ is dense, 
  and $\mc A$ is an almost disjoint family on $K$ of cardinality $\lambda$. Then 
  $\ell_\infty(D)/C(K)$ contains a copy of $c_0(\lambda)$. 
\end{proposition}

\begin{proof}
  Let $\langle U_\alpha \mid \alpha < \lambda \rangle$ be an injective enumeration of $\mc A$ 
  and, for each $\alpha < \lambda$, let $f_\alpha \in \ell_\infty(D)$ be the characteristic 
  function of $D \cap U_\alpha$. For an arbitrary $g \in \ell_\infty(D)$, let 
  $[g]$ denote the equivalence class of $g$ in $\ell_\infty(D)/C(K)$.
  
\begin{claim}
  For all $\alpha < \lambda$, we have $\| [f_\alpha] \| = 1/2$.
\end{claim}

\begin{proof}
  Fix $\alpha < \lambda$. Clearly $\| [f_\alpha] \| \leq 1/2$, 
  since the range of $f_\alpha - 1/2$ lies in the interval $[-1/2, 1/2]$. 
  Suppose for the sake of contradiction that $\| [f_\alpha] \| < 1/2$. Then there is $\varepsilon > 0$ and 
  $\varphi \in C(K)$ such that, for all $x \in D$, we have $f_\alpha(x) + \varphi(x) \in [\varepsilon, 1]$.
  Since $U_\alpha$ is not compact and hence not closed, we can find $x \in \cl(U_\alpha) \setminus U_\alpha$. Since 
  $U_\alpha$ is a regular open set, we have $x \notin \mathrm{int}(\cl(U_\alpha))$. In particular, 
  we have $x \in \cl(D \cap U_\alpha) \cap \cl(D \setminus U_\alpha)$. 
  
  Using the continuity of $\varphi$, fix an open neighborhood $U$ of $x$ such that 
  $|\varphi(y) - \varphi(x)| < \varepsilon/2$ for all $y \in U$. Fix $y_0 \in (U \cap D) \setminus 
  U_\alpha$ and $y_1 \in U \cap D \cap U_\alpha$. Since $f_\alpha(y_0) = 0$, we must have 
  $\varphi(y_0) \geq \varepsilon$. Since $f_\alpha(y_1) = 1$, we must have $\varphi(y_1) \leq 0$. 
  But then $|\varphi(y_0) - \varphi(y_1)| \geq \varepsilon$, contradicting our choice of $U$. 
\end{proof}

\begin{claim}
  Suppose that $A \subseteq \lambda$ is finite and $\langle r_\alpha \mid 
  \alpha \in A \rangle$ is a sequence of real numbers. Then
  \[
    \left\| \sum_{\alpha \in A} r_\alpha [f_\alpha] \right\| \leq 
    \sup_{\alpha \in A} |r_\alpha|.
  \]
\end{claim}

\begin{proof}
  By the previous claim, we can assume that $|A| > 1$.
  Let $[A]^{\geq 2}$ denote the collection of all subsets of $A$ of 
  cardinality at least $2$. For all $B \in [A]^{\geq 2}$, let
  \[
    U_B := \bigcap_{\beta \in B} U_\beta \cap \bigcap_{\alpha \in A \setminus B} K \setminus U_\alpha.
  \]
  Note that, if $B_0$ and $B_1$ are distinct elements of $[A]^{\geq 2}$, then $U_{B_0} \cap 
  U_{B_1} = \emptyset$.  
  
  \begin{subclaim}
    For all $B \in [A]^{\geq 2}$, $U_B$ is clopen.
  \end{subclaim}
  
  \begin{proof}
    Let $U_B^* := \bigcap_{\beta \in B} U_\beta$. By the definition of an almost disjoint family on $K$, 
    we know that $U_B^*$ is clopen and, for all $\alpha \in A \setminus B$, $U_B^* \cap U_\alpha$ 
    is clopen. Note that $U_B$ can be rewritten as 
    \[
      U_B^* \setminus \left (\bigcup_{\alpha \in A \setminus B} U_B^* \cap U_\alpha \right ).
    \]  
    The subclaim follows immediately.
  \end{proof}
  Let $f_A := \sum_{\alpha \in A} r_\alpha f_{\alpha}$, and let 
  $r^* = \sup_{\alpha \in A} |r_\alpha|$. We must show that 
  $\|[f_A]\| \leq r^*$. To this end, define a function 
  $\varphi \in C(K)$ as follows. Fix $x \in K$. If there is 
  $B \in [A]^{\geq 2}$ such that $x \in U_B$, note that such a $B$ is unique, 
  and set $\varphi(x) := \sum_{\beta \in B} r_\beta$. If there is no such 
  $B$, then set $\varphi(x) := 0$. The fact that $\varphi$ is continuous 
  follows from the fact that $U_B$ is clopen for all $B \in [A]^{\geq 2}$. 
  We show that the range of $f_A - \varphi$ lies in the interval 
  $[-r^*,r^*]$, which will establish the claim. To this end, fix $x \in K$. 
  If there is $B \in [A]^{\geq 2}$ 
  such that $x \in U_B$, then we have $f_A(x) = \sum_{\beta \in B} r_\beta 
  = \varphi(x)$, so $f_A(x) - \varphi(x) = 0$. If there is a unique 
  $\alpha \in A$ such that $x \in U_\alpha$, then we have 
  $f_A(x) = r_\alpha$ and $\varphi(x) = 0$. Finally, in all other cases 
  we have $f_A(x) = \varphi(x) = 0$. 
\end{proof}

Now Theorem \ref{thm: rosenthal_2}, applied to 
$\langle 2[f_\alpha] \mid \alpha < \lambda \rangle$ and $C = 2$, yields the existence 
of a $\Gamma' \subseteq \lambda$ such that $|\Gamma'| = \lambda$ and 
$\langle f_\alpha \mid \alpha \in \Gamma' \rangle$ generates a copy of 
$c_0(\Gamma')$ in $\ell_\infty(D)/C(K)$.
\end{proof}

We are now ready to prove Theorem B:

\begin{proof}[Proof of Theorem B]
  Suppose that $K$ is a compact Hausdorff space, $D \subseteq K$ is dense, 
  and there is an almost disjoint family on $K$ of cardinality $\lambda$ such 
  that $2^\lambda > 2^{|D|}$. By Proposition \ref{prop: c_0_copy}, 
  $\ell_\infty(D)/C(K)$ contains a copy of $c_0(\lambda)$. If 
  $\ell_\infty(D)/C(K)$ were injective, then Theorem \ref{thm: rosenthal} 
  would imply that $\ell_\infty(D)/C(K)$ contains a copy of $\ell_\infty(\lambda)$. 
  But we have
  \[
    |\ell_\infty(\lambda)| = 2^\lambda > 2^{|D|} = |\ell_\infty(D)/C(K)|,
  \]
  yielding the desired contradiction.
\end{proof}

\section{An almost disjoint family on $\mathbb{N}^*$} \label{section: n-star}

In this section, we will prove Theorem C, stating that $\mathfrak{b} = 2^{\aleph_0}$ 
implies the existence of an almost disjoint family on $\mathbb{N}^*$ of 
cardinality $2^{\aleph_1}$. Recall that $\mathfrak{b}$ denotes the 
\emph{bounding number}, i.e., the least cardinality of a subset 
of ${^{\omega}}\omega$ that is unbounded in the order 
$({^{\omega}}\omega, <^*)$, where $<^*$ denotes eventual domination. 
Together with Theorem B, this
will then yield Theorem A. We will establish Theorem C via a more technical 
combinatorial result, stated below. Given distinct $\sigma, \tau \in {^{\leq \omega_1}}\omega$, we let $\sigma \wedge \tau$ denote the longest 
common initial segment of $\sigma$ and $\tau$. More precisely,
\[
  \sigma \wedge \tau = \bigcup\{\sigma \restriction \eta + 1 \mid \eta \in \dom(\sigma) 
  \cap \dom(\tau) \text{ and } \sigma \restriction \eta+1 = \tau \restriction \eta+1\}.
\]

\begin{theorem} \label{thm: labeled_tree}
  Suppose that $\mathfrak{b} = 2^{\aleph_0}$. Then there is a sequence 
  $\langle x_\sigma \mid \sigma \in 
  {^{<\omega_1}}\omega \rangle$ of elements of $[\omega]^\omega$ such that
  \begin{enumerate}
    \item for all $\sigma \subsetneq \tau \in {^{<\omega_1}}\omega$, we have 
    $x_\sigma \subsetneq^* x_\tau$;
    \item for all distinct $\sigma, \tau \in {^{<\omega_1}}\omega$, we have 
    $x_\sigma \cap x_\tau =^* x_{\sigma \wedge \tau}$;
    \item for every $y \in [\omega]^\omega$ and every $f \in {^{\omega_1}}\omega$, one of the following holds:
    \begin{enumerate}
      \item there is $\eta < \omega_1$ such that
      \[
        y \subseteq^* x_{f \restriction \eta};
      \]
      \item there is an infinite $y' \subseteq y$ such that, for 
      every $\eta < \omega_1$, we have $y' \cap x_{f \restriction \eta} =^* \emptyset$.
    \end{enumerate}
  \end{enumerate}
\end{theorem}

We first show how Theorem \ref{thm: labeled_tree} yields Theorem C.

\begin{proof}[Proof of Theorem C]
  Let $\langle x_\sigma \mid \sigma \in {^{<\omega_1}}\omega \rangle$ be as in 
  Theorem \ref{thm: labeled_tree}. Recall that $\bb{N}^*$ can be seen as the space 
  of all nonprincipal ultrafilters over $\omega$, and a clopen base for $\bb{N}^*$ is 
  given by the sets
  \[
    N_y := \{D \in \bb{N}^* \mid y \in D\}
  \]
  for $y \in [\omega]^\omega$.
  
  For each $f \in {^{\omega_1}}\omega$, let 
  \[
    U_f := \{D \in \bb{N}^* \mid \exists \eta < \omega_1 \ [x_{f \restriction \eta} \in 
    D]\}.
  \]
  We claim that $\{U_f \mid f \in {^{\omega_1}}\omega\}$ is an almost disjoint family 
  on $\bb{N}^*$. We first verify condition (2) in Definition \ref{ad_def}. Fix distinct 
  $f,g \in {^{\omega_1}}\omega$, and let $\eta < \omega_1$ be least such that $f(\eta) \neq 
  g(\eta)$. Then, by construction, we have $U_f \cap U_g = N_{f \restriction \eta}$, 
  which is clopen in $\bb{N}^*$ and hence compact.
  
  We now turn to verifying condition (1) in Definition \ref{ad_def}. Fix 
  $f \in {^{\omega_1}}\omega$. Since $U_f = \bigcup \{N_{x_{f \restriction \eta}} 
  \mid \eta < \omega_1\}$ is a union of open sets, $U_f$ is open itself. 
  To verify that $U_f$ is a \emph{regular} open set, we must show that no element 
  of $\bb{N}^* \setminus U_f$ is in the interior of the closure of $U_f$. In other words, 
  it suffices to show that, 
  for every $D \in \bb{N}^* \setminus U_f$ and every open neighborhood $W$ of $D$, 
  we can find an open set $W' \subseteq W$ disjoint from $U_f$. Fix such a 
  $D$ and $W$. By shrinking $W$ if necessary, we can assume that $W = N_y$ for 
  some set $y \in [\omega]^\omega$. 
  
  Now consider condition (3) in the statement of 
  Theorem \ref{thm: labeled_tree} applied to $y$ and $f$. If condition (3)(a) 
  is satisfied, then there is $\eta < \omega_1$ such that $y \subseteq^* 
  x_{f \restriction \eta}$. But then it would follow that $x_{f \restriction 
  \eta} \in D$, and hence $D \in U_f$.
  
  Hence, condition (3)(b) is satisfied, so we can find an infinite 
  $y_0 \subseteq y$ such that, for every $\eta < \omega_1$, we have 
  $y_0 \cap x_{f \restriction \eta} =^* \emptyset$. But then 
  $W' := N_{y_0} \subseteq W$ and $W' \cap U_f = \emptyset$, as desired.
  
  Finally, to verify that $U_f$ is not compact, note $U_f = 
  \bigcup \{N_{f\restriction \eta} \mid \eta < \omega_1\}$ and 
  $\langle N_{f\restriction \eta} \mid \eta < \omega_1 \rangle$ is a strictly 
  $\subsetneq$-increasing sequence of open sets. It follows that $U_f$ is not 
  even Lindel\"{o}f, let alone compact.
\end{proof}

\begin{proof}[Proof of Theorem \ref{thm: labeled_tree}]
  Using the fact that $\mathfrak{b} = 2^{\aleph_0}$, we begin by fixing the following data:
  \begin{itemize}
    \item a $\subsetneq^*$-increasing sequence $\langle a_\eta \mid \eta < 2^{\aleph_0} \rangle$ 
    of elements of $[\omega]^\omega$;
    \item an enumeration $\langle \sigma_\xi \mid \xi < 2^{\aleph_0} \rangle$ of 
    ${^{<\omega_1}}\omega$ 
    such that, for all $\sigma \subsetneq \tau$ in ${^{<\omega_1}}\omega$, $\sigma$ is 
    enumerated before $\tau$;
    \item an enumeration $\langle y_\zeta \mid \zeta < 2^{\aleph_0} \rangle$ of $[\omega]^\omega$.
  \end{itemize}
  We will construct a sequence $\vec{x} = \langle x_\sigma \mid \sigma \in {^{<\omega_1}}\omega 
  \rangle$ as in the statement of the theorem. More precisely, we will recursively construct the following 
  objects:
  \begin{itemize}
    \item a $\subseteq$-increasing sequence $\langle \Sigma_\alpha \mid \alpha < 2^{\aleph_0} 
    \rangle$ of $\subseteq$-downward-closed subsets of ${^{<\omega_1}}\omega$ such that, 
    for all $\alpha < 2^{\aleph_0}$, we have $|\Sigma_\alpha| < 2^{\aleph_0}$ and 
    $\sigma_\alpha \in \Sigma_\alpha$;
    \item an increasing sequence $\langle \eta_\alpha \mid \alpha < 2^{\aleph_0} \rangle$ of 
    ordinals less than $2^{\aleph_0}$; and
    \item sequences $\langle \vec{x}^\alpha \mid \alpha < 2^{\aleph_0} \rangle$ such that, 
    for all $\alpha < \beta < 2^{\aleph_0}$:
    \begin{itemize}
      \item $\vec{x}^\alpha = \langle x_\sigma \mid \sigma \in \Sigma_\alpha \rangle$;
      \item for all $\sigma \in \Sigma_\alpha$, $x_\sigma \subseteq^* a_{\eta_\alpha}$;
      \item $\vec{x}^\alpha$ satisfies requirements (1) and (2) in the statement of the 
      theorem;
      \item $\vec{x}^\alpha = \vec{x}^\beta \restriction \Sigma_\alpha$.
    \end{itemize}
  \end{itemize}
  Fix $\beta < 2^{\aleph_0}$ and suppose that we have constructed $\langle (\Sigma_\alpha, 
  \eta_\alpha, \vec{x}^\alpha) \mid \alpha < \beta \rangle$. Let 
  $\Sigma^* = \bigcup \{\Sigma_\alpha \mid \alpha < \beta\}$, $\eta^* = \sup\{\eta_\alpha 
  \mid \alpha < \beta\}$, and $\vec{x}^* = \bigcup \{\vec{x}^\alpha \mid \alpha < \beta\}$ 
  (in particular, all three are $\emptyset$ if $\beta = 0$). Note that, since 
  $2^{\aleph_0} = \mathfrak{b}$, it follows that $2^{\aleph_0}$ is regular, and hence 
  we have $|\Sigma^*|, \eta^* < 2^{\aleph_0}$.
  
  Consider the set $y_\beta \in [\omega]^\omega$. Our construction now splits into 
  the following two cases:
  \begin{itemize}
    \item \textbf{Case 1.} One of the following three situations holds:
    \begin{itemize}
      \item \textbf{Case 1a.} There is $\sigma \in \Sigma^*$ such that 
      $y_\beta \subseteq^* x_\sigma$.
      \item \textbf{Case 1b.} There are $\sigma, \sigma' \in \Sigma^*$ such that 
      \begin{itemize}
        \item $y_\beta \cap x_\sigma \not\subseteq^* x_{\sigma'}$; and
        \item $y_\beta \cap x_{\sigma'} \not\subseteq^* x_\sigma$.
      \end{itemize}
      \item \textbf{Case 1c.} There is an infinite set $y'_\beta \subseteq y_\beta$ such 
      that, for all $\eta < 2^{\aleph_0}$, we have $y'_\beta \cap a_\eta =^* \emptyset$.
    \end{itemize}     
    \item \textbf{Case 2.} Case 1 does not hold.
  \end{itemize}
  Suppose first that we are in case 1. Let $\xi < 2^{\aleph_0}$ be least such that $\sigma_\xi 
  \notin \Sigma^*$, and let $\Sigma_\beta = \Sigma^* \cup \{\sigma_\xi\}$. Note that 
  $\Sigma_\beta$ remains $\subseteq$-downward-closed due to the minimality of $\xi$. We must 
  specify $\eta_\beta$ and $x_{\sigma_\xi}$.
  
  Suppose first that $\dom(\sigma_\xi)$ is a successor ordinal, say $\dom(\sigma_\xi) = 
  \gamma + 1$. Set $\eta_\beta = \eta^* + 1$. Then set 
  \[
    x_{\sigma_\xi} = x_{\sigma_\xi \restriction \gamma} \cup \{a_{\eta_\beta} \setminus a_{\eta^*}\}.
  \]
  For all $\sigma \in \Sigma^*$, we have $\sigma \wedge \sigma_\xi = \sigma \wedge (\sigma_\xi 
  \restriction \gamma)$ and, since $x_\sigma \subseteq^* a_{\eta^*}$, we have 
  \[
    x_\sigma \cap x_{\sigma_\xi} =^* x_\sigma \cap x_{\sigma_\xi \restriction \gamma} =^* 
    x_{\sigma \wedge (\sigma_\xi \restriction \gamma)} = x_{\sigma \wedge \sigma_\xi}.
  \]
  We have thus satisfied all of the necessary requirements of the construction.
  
  Suppose next that $\dom(\sigma_\xi) = \gamma$ is a limit ordinal. Set $\eta_\beta = \eta^*$. 
  Let $\langle \gamma_n \mid 
  n < \omega \rangle$ be a strictly increasing sequence of ordinals converging to $\gamma$.
  For each $\sigma \in \Sigma^*$, let $h_\sigma \in {^{\omega}}\omega$ be an increasing function 
  such that, for all $n < \omega$, we have
  \[
    x_\sigma \cap x_{\sigma_\xi \restriction \gamma_n} \cap [h_\sigma(n), \omega) = 
    x_{\sigma \wedge (\sigma_\xi \restriction \gamma_n)} \cap [h_\sigma(n),\omega).
  \]
  Also fix an increasing function $h \in {^{\omega}}\omega$ such that, for all 
  $n < \omega$, we have
  \[
    x_{\sigma_\xi \restriction \gamma_n} \cap [h(n),\omega) \subseteq a_{\eta^*}.
  \]
  Since $\mathfrak{b} = 2^{\aleph_0}$, we can find a single function $h^* \in {^{\omega}}\omega$ 
  such that $h < h^*$ and $h_\sigma <^* h^*$ for all $\sigma \in \Sigma^*$. Now set
  \[
    x_{\sigma_\xi} = \bigcup\{x_{\sigma_\xi \restriction \gamma_n} \cap 
    [h^*(n),\omega) \mid n < \omega\}.
  \]
  It is clear from the construction that we continue to satisfy requirement (1) in 
  the statement of the theorem, and the fact that $h <^* h^*$ implies that 
  $x_{\sigma_\xi} \subseteq^* a_{\eta^*}$.
  
  Let us now verify that we continue to satisfy requirement (2) in the statement of the 
  theorem. To this end, fix $\sigma \in \Sigma^*$, and fix $n_\sigma < \omega$ such 
  that 
  \begin{itemize}
    \item $\sigma \wedge \sigma_\xi = \sigma \wedge (\sigma_\xi \restriction \gamma_{n_\sigma})$; and
    \item for all $n \in [n_\sigma,\omega)$, we have $h_\sigma(n) < h^*(n)$.
  \end{itemize}
  We must show that
  \[
    x_\sigma \cap x_{\sigma_\xi} =^* x_{\sigma \wedge \sigma_\xi} = 
    x_{\sigma \wedge (\sigma_\xi \restriction \gamma_{n_\sigma})}.
  \]
  The inclusion $x_\sigma \cap x_{\sigma_\xi} \prescript{*}{}\supseteq \
  x_{\sigma \wedge \sigma_\xi}$ is immediate from the construction. For the other 
  direction, note that, for all $n \in [n_\sigma,\omega)$, we have
  \[
    x_\sigma \cap x_{\sigma_\xi \restriction \gamma_n} \cap [h^*(n),\omega) = 
    x_{\sigma \wedge (\sigma_\xi \restriction \gamma_n)} \cap [h^*(n),\omega).
  \]  
  Therefore, any elements of $x_\sigma \cap x_{\sigma_\xi}$ that are \emph{not} 
  in $x_{\sigma \cap \sigma_\xi}$ must come from $x_{\sigma_\xi \restriction \gamma_n}$ 
  for some $n < n_\sigma$. But, for any such $n$, we have
  \[
    x_\sigma \cap x_{\sigma_\xi \restriction \gamma_n} \subseteq^* 
    x_{\sigma \wedge (\sigma_\xi \restriction \gamma_n)} \subseteq^* x_{\sigma 
    \wedge \sigma_\xi},
  \]
  so there are only finitely many such elements. This completes case 1 of step $\beta$.
  
  Suppose now that we are in case 2. There are two subcases to consider:
  \begin{itemize}
    \item \textbf{Case 2a.} There is $\sigma^* \in \Sigma^*$ such that 
    $(y_\beta \cap x_\sigma) \subseteq^* (y_\beta \cap x_{\sigma^*})$ for all 
    $\sigma \in \Sigma^*$. If this holds, then, since we are not in case 1a, 
    the set $y_\beta \setminus x_{\sigma^*}$ must be infinite.
    \item \textbf{Case 2b.} Otherwise, using the fact that we are in neither case 1b 
    nor 2a, we can find a $\subseteq$-increasing sequence $\langle \sigma_n 
    \mid n < \omega \rangle$ from $\Sigma^*$ such that 
    \begin{itemize}
      \item $\langle y_\beta \cap x_{\sigma_n} \mid n < \omega \rangle$ is 
      $\subsetneq^*$-increasing;
      \item for all $\sigma \in \Sigma^*$, there is $n < \omega$ such that 
      $y_\beta \cap x_\sigma \subseteq^* x_{\sigma_n}$.
    \end{itemize}
  \end{itemize}
  In either subcase, we can find an infinite $y'_\beta \subseteq y_\beta$ such 
  that $y'_\beta \cap x_\sigma =^* \emptyset$ for all $\sigma \in \Sigma^*$. Since we 
  are not in case 1c, we can find $\eta_\beta > \eta^*$ such that 
  $|y'_\beta \cap a_{\eta_\beta}| = \aleph_0$. By thinning out $y'_\beta$ if necessary, we 
  can assume that $y'_\beta \subseteq a_{\eta_\beta}$.
  
  Let $\xi < 2^{\aleph_0}$ be least such that $\sigma_\xi \notin \Sigma^*$, and let 
  \[
    \Sigma_\beta = \Sigma^* \cup \{\sigma_\xi\} \cup \{\sigma_\xi{}^\frown \langle \ell \rangle 
    \mid \ell < \omega \}.
  \]
  Let $y'_\beta = \bigcup \{y'_{\beta,\ell} \mid \ell < \omega\}$ be a partition of $y'_\beta$ 
  into infinite, pairwise disjoint sets.
  
  We must specify $x_{\sigma_\xi}$ and $x_{\sigma_\xi{}^\frown \langle \ell \rangle}$ for all 
  $\ell < \omega$. Suppose first that $\dom(\sigma_\xi)$ is a successor ordinal, say 
  $\dom(\sigma_\xi) = \gamma + 1$. Then set
  \[
    x_{\sigma_\xi} = x_{\sigma_\xi \restriction \gamma} \cup y'_{\beta,0}
  \]
  and, for $\ell < \omega$, set
  \[
    x_{\sigma_\xi{}^\frown \langle \ell \rangle} = x_{\sigma_\xi} \cup y'_{\beta, \ell+1}.
  \]
  It is clear that this satisfies all of the requirements of the construction.
  
  Suppose next that $\dom(\sigma_\xi) = \gamma$ is a limit ordinal. Let $\langle \gamma_n \mid 
  n < \omega \rangle$ be a strictly increasing sequence of ordinals converging to $\gamma$. 
  For $\sigma \in \Sigma^*$, define $h_\sigma \in {^{\omega}}\omega$ exactly as in case 1. 
  Let $h \in {^{\omega}}\omega$ be an increasing function such that, for all $n < \omega$, 
  we have 
  \begin{itemize}
    \item $x_{\sigma_\xi \restriction \gamma_n} \cap [h(n),\omega) \subseteq a_{\eta^*}$; and
    \item $x_{\sigma_\xi \restriction \gamma_n} \cap [h(n),\omega) \cap y'_\beta = \emptyset$.
  \end{itemize}    
  As in case 1, find $h^* \in {^{\omega}}\omega$ such that $h < h^*$ and $h_\sigma <^* h^*$ 
  for all $\sigma \in \Sigma^*$, and set 
  \[
    x_{\sigma_\xi} = \bigcup \{x_{\sigma_\xi \restriction \gamma_n} \cap 
    [h^*(n),\omega) \mid n < \omega\}.
  \]
  Then, for all $\ell < \omega$, set
  \[
    x_{\sigma_\xi{}^\frown \langle \ell \rangle} = x_{\sigma_\xi} \cup y'_{\beta,\ell+1}.
  \]
  The verification that we continue to satisfy the recursion requirements is exactly as in 
  case 1, so we leave it to the reader.
  
  This completes the description of the construction of $\vec{x}$. It remains to verify that 
  the completed sequence satisfies item (3) in the statement of the theorem. To this end, 
  fix a $y \in [\omega]^\omega$ and $f \in {^{\omega_1}}\omega$.
  Fix $\beta < \omega_1$ such that $y = y_\beta$.
  
  Consider stage $\beta$ of the construction, and let $\Sigma^*$ and $\eta^*$ be 
  as defined during that stage. Suppose first that we were in case 1a 
  during stage $\beta$. Let $\sigma \in \Sigma^*$ be $\subseteq$-minimal with 
  $y_\beta \subseteq^* x_\sigma$. If $\sigma \subseteq f$, then (3)(a) in the 
  statement of the theorem is satisfied for this choice of $y$ and $f$. 
  Otherwise, $y' := y_\beta \setminus x_{\sigma \wedge f}$ witnesses this instance 
  of (3)(b).
  
  Suppose next that we were in case 1b, as witnessed by $\sigma, \sigma' \in \Sigma^*$. 
  Note that $\sigma$ and $\sigma'$ must be $\subseteq$-incomparable, so we can assume 
  without loss of generality that $\sigma \not\subseteq f$. Then 
  $y' := (y_\beta \cap x_\sigma) \setminus x_{\sigma \wedge f}$ witnesses this instance 
  of (3)(b).
  
  If we were in case 1c, then the $y'_\beta$ witnessing case 1c will also witness this 
  instance of (3)(b).
  
  Finally, suppose that we were in case 2 during stage $\beta$. Let $\xi < 2^{\aleph_0}$ 
  be least such that $\sigma_\xi \notin \Sigma^*$, and find $\ell < \omega$ such 
  that $\sigma_\xi{}^\frown \langle \ell \rangle \not\subseteq f$. Then 
  $y'_{\beta,\ell+1}$ witnesses this instance of (3)(b), completing the proof of the 
  theorem.
\end{proof}

We can finally establish Theorem A.

\begin{proof}[Proof of Theorem A]
  Suppose that $\mathsf{CH}$ holds. In particular, we have $\mathfrak{b} = 2^{\aleph_0}$. 
  By Theorem C, there is an almost disjoint 
  family on $\mathbb{N}^*$ of cardinality $2^{\aleph_1}$. Since the density 
  of $\mathbb{N}^*$ is $2^{\aleph_0} = \aleph_1$, and since 
  $2^{\aleph_1} < 2^{2^{\aleph_1}}$, Theorem B then implies that 
  $\ell^\infty([\omega]^\omega)/C(\mathbb{N}^*)$ is not injective. 
  Recalling the commutative diagram in the Introduction right above the statement 
  of Theorem A, we know that $\ell_\infty([\omega]^\omega)/C(\mathbb{N}^*)$ 
  appears as $c\kappa^2(c_0)$ in a commutative diagram yielding an injective 
  resolution of $c_0$. It then follows from the definition of injective 
  dimension that $\mathrm{id}(c_0) \geq 3$.
\end{proof}

\section{Open questions} \label{section: questions}

We end by (re)stating a few open questions. Of course, the question that 
is of most interest to us here remains the motivating question of this article.

\begin{question} \label{question: id}
  What is $\mathrm{id}(c_0)$? Is it provable in $\ZFC$ that $\mathrm{id}(c_0) = 
  \infty$?
\end{question}

For a Banach space $X$, the assertion that $\mathrm{id}(X) \leq n$ is equivalent 
to the assertion that the group $\mathrm{Ext}^{n+1}(Y,X)$ is trivial for all 
Banach spaces $Y$ (cf.\ \cite{sanchez_homological}), so Question 
\ref{question: id} can be rephrased as asking whether, for all positive 
$n < \omega$, there exists a Banach space $Y$ such that 
$\mathrm{Ext}^n(Y,c_0) \neq 0$.
We are also interested in more specific questions about particular 
$\mathrm{Ext}$ groups involving $c_0$. For example, the following question is asked 
in \cite{sanchez_homological}.

\begin{question} \label{ext_2_question}
  Is it consistent with $\ZFC$ that $\mathrm{Ext}^2(c_0(\omega_1), c_0) = 0$?
\end{question}

Concretely, this reduces to the following question: is it consistent that every 
continuous linear operator $f:c_0(\omega_1) \ra \ell_\infty([\omega]^\omega)/
C(\mathbb{N}^*)$ lifts to a continuous linear operator 
$F:c_0(\omega_1) \ra \ell_\infty([\omega]^\omega)$?
In \cite[Theorem 1]{corrigendum}, it is proven that $\CH$ implies that 
$\mathrm{Ext}^2(c_0(\omega_1), c_0) \neq 0$. It is at least conceivable that 
the opposite answer might follow from an assumption such as the Proper Forcing 
Axiom.

Finally, we wonder whether an assumption like $\mathfrak{b} = 2^{\aleph_0}$ is necessary 
to construct a large almost disjoint family on $\bb{N}^*$.

\begin{question}
  Is it provable in $\ZFC$ that there exists an almost disjoint family on 
  $\bb{N}^*$ of cardinality $2^{\aleph_1}$, or even of cardinality 
  $2^{2^{\aleph_0}}$?
\end{question}

\bibliographystyle{plain}
\bibliography{bib}

@article {sanchez_homological,
    AUTHOR = {Cabello S\'anchez, F. and Castillo, J. M. F. and Garc\'ia, R.},
     TITLE = {Homological dimensions of {B}anach spaces},
   JOURNAL = {Mat. Sb.},
  FJOURNAL = {Matematicheski\u i\ Sbornik},
    VOLUME = {212},
      YEAR = {2021},
    NUMBER = {4},
     PAGES = {91--112},
      ISSN = {0368-8666,2305-2783},
   MRCLASS = {46M18 (18G15)},
  MRNUMBER = {4236253},
MRREVIEWER = {Sergey\ V.\ Astashkin},
       DOI = {10.4213/sm9425},
       URL = {https://doi.org/10.4213/sm9425},
}

@book {separably_injective,
    AUTHOR = {Avil\'es, Antonio and S\'anchez, F\'elix Cabello and Castillo,
              Jes\'us M. F. and Gonz\'alez, Manuel and Moreno, Yolanda},
     TITLE = {Separably injective {B}anach spaces},
    SERIES = {Lecture Notes in Mathematics},
    VOLUME = {2132},
 PUBLISHER = {Springer, [Cham]},
      YEAR = {2016},
     PAGES = {xxii+217},
      ISBN = {978-3-319-14740-6; 978-3-319-14741-3},
   MRCLASS = {46-02 (03C20 03E10 46B20)},
  MRNUMBER = {3469461},
MRREVIEWER = {Antonis\ N.\ Manoussakis},
       DOI = {10.1007/978-3-319-14741-3},
       URL = {https://doi.org/10.1007/978-3-319-14741-3},
}

@misc{CS1,
      title={Lectures on Condensed Mathematics}, 
      author={Peter Scholze},
      year={2026},
      eprint={2605.03658},
      archivePrefix={arXiv},
      primaryClass={math.NT},
      url={https://arxiv.org/abs/2605.03658}, 
      note={arXiv preprint: 2605.03658}
}

@misc{complex,
      title={Condensed Mathematics and Complex Geometry}, 
      author={Dustin Clausen and Peter Scholze},
      year={2026},
      eprint={2605.11731},
      archivePrefix={arXiv},
      primaryClass={math.CV},
      url={https://arxiv.org/abs/2605.11731}, 
      note={arXiv preprint: 2605.11731}
}

@article {phillips,
    AUTHOR = {Phillips, R. S.},
     TITLE = {On linear transformations},
   JOURNAL = {Trans. Amer. Math. Soc.},
  FJOURNAL = {Transactions of the American Mathematical Society},
    VOLUME = {48},
      YEAR = {1940},
     PAGES = {516--541},
      ISSN = {0002-9947,1088-6850},
   MRCLASS = {46.3X},
  MRNUMBER = {4094},
MRREVIEWER = {N.\ Dunford},
       DOI = {10.2307/1990096},
       URL = {https://doi.org/10.2307/1990096},
}

@article {sobczyk,
    AUTHOR = {Sobczyk, Andrew},
     TITLE = {Projection of the space {$(m)$} on its subspace {$(c_0)$}},
   JOURNAL = {Bull. Amer. Math. Soc.},
  FJOURNAL = {Bulletin of the American Mathematical Society},
    VOLUME = {47},
      YEAR = {1941},
     PAGES = {938--947},
      ISSN = {0002-9904},
   MRCLASS = {46.3X},
  MRNUMBER = {5777},
MRREVIEWER = {F.\ J.\ Murray},
       DOI = {10.1090/S0002-9904-1941-07593-2},
       URL = {https://doi.org/10.1090/S0002-9904-1941-07593-2},
}

@article {amir,
    AUTHOR = {Amir, D.},
     TITLE = {Projections onto continuous function spaces},
   JOURNAL = {Proc. Amer. Math. Soc.},
  FJOURNAL = {Proceedings of the American Mathematical Society},
    VOLUME = {15},
      YEAR = {1964},
     PAGES = {396--402},
      ISSN = {0002-9939,1088-6826},
   MRCLASS = {46.10 (46.25)},
  MRNUMBER = {165350},
MRREVIEWER = {T.\ R.\ Jenkins},
       DOI = {10.2307/2034512},
       URL = {https://doi.org/10.2307/2034512},
}

@article {rosenthal,
    AUTHOR = {Rosenthal, Haskell P.},
     TITLE = {On relatively disjoint families of measures, with some
              applications to {B}anach space theory},
   JOURNAL = {Studia Math.},
  FJOURNAL = {Polska Akademia Nauk. Instytut Matematyczny. Studia
              Mathematica},
    VOLUME = {37},
      YEAR = {1970},
     PAGES = {13--36},
      ISSN = {0039-3223,1730-6337},
   MRCLASS = {46.10},
  MRNUMBER = {270122},
MRREVIEWER = {G.\ I.\ Gaudry},
       DOI = {10.4064/sm-37-1-13-36},
       URL = {https://doi.org/10.4064/sm-37-1-13-36},
}

@article {corrigendum,
    AUTHOR = {Avil\'es, A. and Cabello S\'anchez, F. and Castillo, J. M. F.
              and Gonz\'alez, M. and Moreno, Y.},
     TITLE = {Corrigendum to ``{O}n separably injective {B}anach spaces''
              [{A}dv. {M}ath. 234 (2013) 192--216] [{MR}3003929]},
   JOURNAL = {Adv. Math.},
  FJOURNAL = {Advances in Mathematics},
    VOLUME = {318},
      YEAR = {2017},
     PAGES = {737--747},
      ISSN = {0001-8708,1090-2082},
   MRCLASS = {46A22 (46B04 46B08 46B26)},
  MRNUMBER = {3689754},
       DOI = {10.1016/j.aim.2017.08.012},
       URL = {https://doi.org/10.1016/j.aim.2017.08.012},
}

@unpublished{scholze_email,
  author = {Scholze, Peter},
  note = {Personal communication},
  year = {2019}
  }

@incollection {hrusak_ad,
    AUTHOR = {Hru\v{s}\'{a}k, Michael},
     TITLE = {Almost disjoint families and topology},
 BOOKTITLE = {Recent progress in general topology. {III}},
     PAGES = {601--638},
 PUBLISHER = {Atlantis Press, Paris},
      YEAR = {2014},
      ISBN = {978-94-6239-023-2; 978-94-6239-024-9},
   MRCLASS = {03-02 (03E05 03E35 54-02 54H05)},
  MRNUMBER = {3205494},
MRREVIEWER = {J\"org\ D.\ Brendle},
       DOI = {10.2991/978-94-6239-024-9\_14},
       URL = {https://doi.org/10.2991/978-94-6239-024-9_14},
}

@article {whitley,
    AUTHOR = {Whitley, Robert},
     TITLE = {Mathematical {N}otes: {P}rojecting {$m$} onto {$c_0$}},
   JOURNAL = {Amer. Math. Monthly},
  FJOURNAL = {American Mathematical Monthly},
    VOLUME = {73},
      YEAR = {1966},
    NUMBER = {3},
     PAGES = {285--286},
      ISSN = {0002-9890,1930-0972},
   MRCLASS = {99-04},
  MRNUMBER = {1533692},
       DOI = {10.2307/2315346},
       URL = {https://doi.org/10.2307/2315346},
}

\end{document}